\documentclass[12pt]{amsart}
\usepackage[utf8]{inputenc}
\usepackage[]{geometry}

\usepackage[backend=biber, style=alphabetic, sorting=nty]{biblatex}
\usepackage{amsmath}
\usepackage{amsfonts}
\usepackage{amssymb}
\usepackage{amsthm}
\usepackage{amscd}
\usepackage{amstext} %

\usepackage{soul}

\usepackage{mathrsfs}
\usepackage{enumitem}
\usepackage{subfig}

\usepackage[hyphens]{xurl}
\usepackage{hyperref}
\hypersetup{
	colorlinks=false,
	linkcolor=blue,
	filecolor=magenta,      
	urlcolor=cyan,
}
\usepackage[hyphenbreaks]{breakurl}
\RequirePackage{graphicx}
\RequirePackage{tikz}
\usepackage{tikz-cd}
\usetikzlibrary{arrows.meta,decorations.pathmorphing,backgrounds,positioning,fit,petri, intersections, calc, through, angles, quotes}
\RequirePackage{tikz-3dplot}
\RequirePackage{wrapfig}

\newtheorem{thm}{Theorem}
\newtheorem{lem}[thm]{Lemma}
\newtheorem{prop}[thm]{Proposition}
\newtheorem{coro}[thm]{Corollary}

\newtheorem{conj}[thm]{Conjecture}

\theoremstyle{definition}
\newtheorem{eg}[thm]{Example}

\theoremstyle{definition}
\newtheorem{defn}[thm]{Definition}

\numberwithin{thm}{section}

\newcommand{\CC}{\mathbb{C}}

\newcommand{\RR}{\mathbb{R}}
\newcommand{\QQ}{\mathbb{Q}}
\newcommand{\LL}{\mathbb{L}}

\newcommand{\TT}{\mathbb{T}}

\newcommand{\PP}{\mathbb{P}}

\newcommand{\lL}{\mathscr{L}}
\newcommand{\mM}{\mathscr{M}}

\newcommand{\qQ}{\mathscr{Q}}

\newcommand{\eEE}{\mathcal{E}}

\newcommand{\cCC}{\mathcal{C}}
\newcommand{\fFF}{\mathcal{F}}
\newcommand{\hHH}{\mathcal{H}}

\newcommand{\kKK}{\mathcal{K}}
\newcommand{\oOO}{\mathcal{O}}
\newcommand{\gGG}{\mathcal{G}}

\newcommand{\mMM}{\mathcal{M}}

\newcommand{\pPP}{\mathcal{P}}

\newcommand{\cofib}[1]{\mathsf{Q}(#1)}
\newcommand{\gl}[1]{\mathfrak{gl}_{#1}}

\newcommand{\Hom}{\text{Hom}}

\newcommand{\RHom}{\mathbb{R}\text{Hom}}

\newcommand{\cdga}{\text{cdga}^{\leq 0}}
\newcommand{\dquot}{\textbf{Quot}}
\newcommand{\map}[1]{\text{Map}_{\text{#1}}}

\newcommand{\dspec}{\textbf{Spec}}
\newcommand{\spec}{\text{Spec}}
\newcommand{\rperf}{\textbf{Perf}}
\newcommand{\rpsperf}{\textbf{PsPerf}}
\newcommand{\rmap}{\textbf{Map}}
\newcommand{\infgpd}{\infty\text{-Grpd}}
\newcommand{\rmor}{\textbf{Mor}}
\newcommand{\rpsmor}{\textbf{PsMor}}
\newcommand{\dgcat}{\text{dgCat}}

\newcommand{\supp}{\mbox{supp}}

\title{Derived quot schemes}
\author{Nachiketa Adhikari}
\address{Department of Mathematics, University of Illinois at Urbana-Champaign, Urbana IL 61801, USA}
\email{na17@illinois.edu}
\thanks{This research is supported in part by NSF grants DMS-1802242 and DMS-2201203.}

\begin{document}
	\begin{abstract}
		We define a derived enhancement of the classical quot functor of quotients associated to a coherent sheaf on a nonsingular quasiprojective variety. We prove its representability and show that it has the expected tangent complex. The derived quot scheme of points can be covered by affine charts obtained as spectra of commutative graded differential algebras (cdgas) and we compute an example. As a demonstration of the usefulness of this presentation, we write down an explicit shifted form on this example and conjecture that it agrees with the shifted symplectic structure on the derived stack of perfect complexes constructed by Pantev-Toën-Vaquié-Vezzosi (\cite{ptvv}) and Brav-Dyckerhoff (\cite{bravdyck}).
	\end{abstract}
	\maketitle
	\section{Introduction}
	A (classical) moduli problem over $k = \CC$ is a functor
	\[ \{\text{commutative $k$-algebras}\} \to \{\text{sets}\}. \]
	A \textit{derived moduli problem} is a functor
	\[ \{\text{commutative graded differential $k$-algebras (cdgas)}\} \to \{\text{spaces} \}. \]
	A derived moduli problem $\bf F$ is a \textit{derived enhancement} of a classical moduli problem $F$ when the space that $\bf F$ assigns to an arbitrary commutative $k$-algebra $A$ (thought of as a cdga) is homotopy equivalent to $F(A)$, where the latter has been endowed with the discrete topology.
	
	In the framework developed in \cite{hag2}, \textit{affine derived schemes} form the opposite category of cdgas; \textit{derived schemes} are constructed by gluing affine derived schemes. A derived scheme has a global tangent complex. Given a representable classical moduli problem $F$, it is possible, in many cases, to construct a derived enhancement $\bf F$, representable by a derived scheme, such that
	\begin{enumerate}
		\item the tangent space at a $k$-point of $F$ is a cohomology group of the tangent complex of $\bf F$ and
		\item the higher cohomology groups of the tangent complex naturally encode obstructions to smoothness of $F$.
	\end{enumerate}
	In such a situation, the ranks of the tangent spaces of $k$-points of $F$ may jump at singularities, but the tangent complex of $\bf F$ is a global invariant that can be represented by a perfect complex with terms having constant rank. This is a manifestation of the ``hidden smoothness'' philosophy of derived geometry.
	
	Let $X$ be a smooth quasiprojective variety and $\fFF$ a coherent sheaf on $X$. Recall that the classical quot functor $Quot_{\fFF}$ parameterizes quotient sheaves of $\fFF$. In \S 3, we define the derived quot functor $\dquot_\fFF$ as a derived moduli problem. It classifies quotient complexes of sheaves (appropriately defined). Here is our main result about the derived quot functor:
	\begin{thm}
		\begin{enumerate}
			\item The derived quot functor $\dquot_\fFF$ is representable by a derived scheme.
			\item There is a natural map
			\[  Quot_\fFF \to t_0(\dquot_\fFF) \]
			which is an equivalence of ordinary schemes. Here $t_0(\bf G)$, for a derived moduli functor $\bf G$, is its restriction to commutative $k$-algebras. In other words, the derived quot functor is a derived enhancement of the classical quot functor.
			\item The tangent complex of $\dquot_\fFF$ at a $k$-point corresponding to a quotient $\fFF \to \gGG$ with kernel $\kKK \to \fFF$ is naturally equivalent to $\RHom_X(\kKK, \gGG)$.
		\end{enumerate}\label{maintheorem}
	\end{thm}
	A derived version of the quot scheme has been constructed earlier in \cite{ckquot}, but only as a dg-scheme and not as a derived scheme representing a moduli functor. The first definition of a derived quot functor appeared in \cite{aron} and was for projective schemes. Our definition is for quasiprojective varieties and agrees with the former in the case of projective varieties.
	
	The construction of affine charts of the derived quot scheme as spectra of cdgas makes local computations for derived quot schemes tractable in many cases. Here is an interesting application of this feature to enumerative geometry.
	
	Motivic Donaldson-Thomas theory, introduced in \cite{kontsevich} and further developed in \cite{bjm}, is a categorification of the classical Donaldson-Thomas theory which is concerned with virtual counts of stable coherent sheaves on Calabi-Yau threefolds. One of the ingredients in defining motivic DT invariants is the algebraic d-critical locus structure on moduli spaces of coherent sheaves in the sense of \cite{dcritical}.
	
	Let $X$ be a Calabi-Yau threefold. The derived moduli stack of coherent sheaves on $X$, and more generally the derived moduli stack of perfect complexes on $X$, has a canonical $(-1)$-shifted symplectic structure, as shown in \cite{ptvv}. This shifted symplectic structure truncates to an algebraic d-critical locus on the classical moduli stacks. While it was later shown, in \cite{bbj} and \cite{bbbbj}, that such shifted symplectic structures always locally have a (non-unique) ``Darboux form'', the \textit{canonical} structure guaranteed by the main theorem in \cite{ptvv} has been computed in few examples and that too only in the case of the moduli stack of sheaves with constant Hilbert polynomial: for $X = \CC^3$ in \cite{ricolfi} and more generally for toric local Calabi-Yau threefolds in \cite{katzshi}.
	
	The presentation we obtain for the derived quot scheme can simplify computations of the canonical shifted symplectic structure on these derived moduli stacks for a larger class of Calabi-Yau threefolds. We describe the pullback of a shifted symplectic structure to an affine chart of the derived quot scheme and state a conjecture (Conjecture \ref{conjecture}) that it coincides with the canonical one.
	
	\subsection*{Outline of the paper}
	In section 2 we recall some properties of the derived stack of perfect complexes that will be useful later. In section 3 we define the derived quot functor and prove Theorem \ref{maintheorem}. In section 4 we do explicit calculations of the derived quot scheme and obtain a presentation in select cases. Finally, in section 5, we discuss shifted symplectic structures on derived stacks of perfect complexes and describe the conjectured pullback of one to the derived quot scheme of the Fermat quintic: we choose this space to work with because it is an example of a \textit{projective} Calabi-Yau threefold and also because of the Fermat quintic's prominence in the study of mirror symmetry. 
	
	\subsection*{Notation}
	
	We work in the homotopical algebraic geometry setting of \cite[\S 2.2]{hag2}. We work throughout over $k = \CC$ though many results hold true for any algebraically closed field of characteristic 0.
	
	We use cohomological conventions throughout: a \textit{connective} object $E$ is one for which $\pi_i(E) = 0$ for $i < 0$. We will freely interchange $\pi_i(E)$ and $H^{-i}(E)$, so that a connective chain complex or cdga $A$ is one for which $H^i(A) = 0$ for $i > 0$.
	
	For an $\infty$-category $\cCC$, we denote its \textit{core} i.e.\ maximal $\infty$-subgroupoid by $\cCC^\simeq$.
	
	For a cdga $A$ and a scheme $X$, denote $X\times^h \dspec A$ by $\bf X_A$.
	
	We use \textbf{boldface} letters to denote derived schemes/stacks. The classical truncation of a derived scheme $\bf U$ is denoted by $t_0(\bf U)$.
	
	Throughout, $X$ is a smooth quasiprojective variety over $k$.
	
	\subsection*{Acknowledgments}
	I would like to thank my advisor Sheldon Katz for suggesting this project and helping me throughout with ideas, suggestions and critiques. I am also grateful to Aron Heleodoro and Yun Shi for helpful conversations.
	
	\section{The derived stack of perfect complexes}
	\subsection{Perfect and pseudo-perfect objects}
	For a derived scheme $\bf Y$, denote by $QCoh(\bf Y)$ the dg-category of (complexes of) quasicoherent $\oOO_{\bf Y}$-modules. Using the dg-nerve construction \cite[\S 1.3]{HA}
	we can also treat it as a stable $\infty$-category. This category has a symmetric monoidal structure. Using \cite[Proposition 3.6]{benzvi},
	\begin{defn}
		An object of $QCoh(\bf Y)$ is called {\bf perfect} if it is (strongly) dualizable in the above mentioned symmetric monoidal structure.
	\end{defn}
	For ${\bf Y} = \dspec A$, this is equivalent to being perfect as an $A$-dg-module, which is equivalent to being a compact object in the homotopy i.e. derived category $D(A)$. Denote the full $\infty$-subcategory (equivalently, dg-subcategory) of $QCoh(\bf Y)$ consisting of perfect (resp.\ perfect connective) objects by $Perf({\bf Y})$ (resp.\ $Perf({\bf Y})^{\leq 0}$).
	
	For $X$ a smooth quasiprojective variety, the moduli stack of objects in the dg-category $Perf(X)$ is defined in \cite{tv} as
	\begin{align*}
		{\bf\mMM}_{Perf(X)}: \cdga &\to \infgpd\\
		A &\mapsto \map\dgcat(Perf(X)^{op}, Perf(A))
	\end{align*}
	In order to understand this stack in terms of perfect complexes on $X$, we need the notion of a \textit{pseudo-perfect complex}.
	
	For derived schemes $\bf U$ and $\bf V$ and $\eEE \in QCoh(\bf U \times \bf V)$, consider the functor
	\begin{align*}
		\Phi_\eEE: QCoh({\bf U})^{op} &\to QCoh(\bf V)\\
		\fFF &\mapsto \RR q_*(\RR Hom(\RR p^*\fFF, \eEE))
	\end{align*}
	where $p$ and $q$ are the projections to $\bf U$ and $\bf V$ respectively.
	\begin{defn}
		An object $\eEE \in QCoh(\bf U \times \bf V)$ is called \textbf{pseudo-perfect relative to $\bf V$} if the map $\Phi_\eEE$ defined above takes perfect objects to perfect objects.
		For $U$ a scheme and $A$ a cdga, we will simply call an object $\eEE \in QCoh(\bf U_A)$ \textbf{pseudo-perfect} if it is pseudo-perfect relative to $\bf A$.
	\end{defn}
	By \cite[Proposition 2.22]{tv}, $\eEE \in QCoh(\bf U_A)$ is pseudo-perfect if and only if it is pseudo-perfect as a complex on $U_{\spec H^0(A)}$.
	Denote the full $\infty$-subcategory of $QCoh(\bf U_A)$ generated by the pseudo-perfect objects by $PsPerf(\bf U_A)$. We have the following lemma relating $PsPerf(\bf U_A)$ and $Perf(\bf U_A)$:
	\begin{lem}[{\cite[Lemma 2.8]{tv}}]\label{lem:psperf}
		For $U$ a scheme and $A$ a cdga,
		\begin{enumerate}
			\item If $U$ is proper, then $Perf({\bf U_A}) \subset PsPerf(\bf U_A)$.
			\item If $U$ is smooth, then $PsPerf({\bf U_A}) \subset Perf(\bf U_A)$.
		\end{enumerate}
	\end{lem}
	In particular, for our smooth quasiprojective variety $X$, $PsPerf({\bf X_A}) \subset Perf(\bf X_A)$, with equality holding if $X$ is projective. The functor ${\bf\mMM}_{Perf(X)}$ is, by \cite[\S 3]{tv}, equivalent to
	\begin{align*}
		\rpsperf(X): \cdga &\to \infgpd\\
		A &\mapsto PsPerf({\bf X_A})^\simeq
	\end{align*}
	Since the classical quot functor is usually defined in terms of quotients with proper support, we wish to understand the notion of pseudo-perfect complexes in terms of their support. Accordingly, we have
	\begin{defn}
		The \textbf{support} $\supp\; \eEE$ of a complex of sheaves $\eEE \in QCoh({\bf X_A})$ is the union of the supports of the cohomology sheaves $\hHH^i(\eEE)$, considered as a subscheme of $t_0({\bf X_A}) = X \times \spec H^0(A)$. The complex $\eEE$ is said to have \textbf{proper support over $\bf A$} if $\supp\; \eEE$ is proper over $\spec H^0(A)$.
	\end{defn}
	By \cite[Remark 5.5.3]{DAG}, $\supp\; \eEE$ is proper over $H^0(A)$ if and only if $\supp\; \eEE \to \bf A$ is a proper morphism of derived schemes. Since pseudo-perfectness and properness of support can both be checked at the level of classical schemes, the following is a straightforward application of the proof of \cite[Proposition B.1]{efimov}:
	\begin{lem}
		Let $X$ be a smooth quasiprojective variety and $A$ a cdga. Let $\eEE \in QCoh({\bf X_A})$ be a complex of sheaves on $\bf X_A$. Then $\eEE$ is pseudo-perfect relative to $\bf A$ if and only if it is perfect and has proper support over $\bf A$.
	\end{lem}
	Therefore the derived stack $\rpsperf(X)$ is the classifying stack of perfect complexes with proper support on $X$.

	The derived stack of perfect modules, $\rperf$, is defined in \cite[\S 3]{tv} (where it is denoted by $\bf \mMM_1$):
	\begin{align*}
		\rperf: \cdga &\to \infgpd\\
		A &\mapsto Perf({\bf A})^\simeq
	\end{align*}
	while the derived mapping stack $\rperf(X) := \rmap(X, \rperf)$ is
	\begin{align*}
		\rperf(X): \cdga &\to \infgpd\\
		A &\mapsto Perf({\bf X_A})^\simeq
	\end{align*}
	It follows from Lemma \ref{lem:psperf}, that there is a morphism of derived stacks
	\[ \rpsperf(X) \to \rperf(X) \]
	that is an equivalence when $X$ is projective.
	
	\subsection{The derived stack of morphisms}
	
	Let $\fFF$ be a perfect complex on $X$. Define the derived stack of morphisms with source $\fFF$ to be the functor that assigns to a cdga $A$ the core of the undercategory of $\fFF_A := \fFF \otimes \oOO_{\bf A} \in Perf({\bf X_A})$:
	\begin{align*}
		\rmor(\fFF, -): \cdga &\to \infgpd\\
		A &\mapsto Perf({\bf X_A})_{\fFF_A/}^\simeq
	\end{align*}
	We will actually need a refinement of this definition. Define the derived stack of morphisms with source $\fFF$ and target a pseudo-perfect complex as the pullback
	\[\begin{tikzcd}
		{\rpsmor(\fFF, -)} \arrow{r}{t} \arrow{d} & {\rpsperf(X)} \arrow{d}\\
		{\rmor(\fFF, -)} \arrow{r}{t} & {\rperf(X)}
	\end{tikzcd}\]
	
	\begin{lem}
		The derived stack $\rpsmor(\fFF, -)$ is locally geometric.
	\end{lem}
	\begin{proof}
		By \cite[Proposition 3.13]{tv}, we have a decomposition \[\rpsperf(X) = \cup_{a\leq b} \rpsperf(X)^{[a,b]},\] where each component is $n$-geometric for some $n$. Here the subfunctors $\rpsperf(X)^{[a,b]} \subset \rpsperf(X)$ correspond to modules with Tor-amplitude in $[a, b]$. Defining $\rpsmor(\fFF, -)^{[a,b]}$ as the pullback
		\[\begin{tikzcd}
			{\rpsmor(\fFF, -)^{[a,b]}} \arrow{r}{t} \arrow{d} & {\rpsperf(X)^{[a,b]}} \arrow{d}\\
			{\rpsmor(\fFF, -)} \arrow{r}{t} & {\rpsperf(X)}
		\end{tikzcd}\]
		we obtain a similar decomposition for $\rpsmor(\fFF, -)$. It suffices to show that the upper horizontal map
		\[ \rpsmor(\fFF, -)^{[a,b]} \to \rpsperf(X)^{[a,b]} \]
		is $m$-representable for some $m$ i.e. that given a map $\dspec A \to \rpsperf(X)^{[a,b]}$ corresponding to a pseudo-perfect complex $\gGG$ on $\bf X_A$, the pullback $\rmor(\fFF, \gGG)$ of the diagram
		\[\begin{tikzcd}
			{\rmor(\fFF, \gGG)} \arrow{r} \arrow{d} & {\dspec A} \arrow{d}{\gGG}\\
			{\rpsmor(\fFF, -)^{[a,b]}} \arrow{r}{t} & {\rpsperf(X)^{[a,b]}}
		\end{tikzcd}\]
		is $m$-geometric. Abusing notation to denote the map ${\bf X_B} \to {\bf X_A}$ by $f$, the stack $\rmor(\fFF, \gGG)$ is given by
		\begin{align*}
			\rmor(\fFF, \gGG): A/\cdga &\to \infgpd\\
			(f:A \to B) &\mapsto Map_{\bf X_B}(\fFF_B, f^*\gGG)
		\end{align*}
		which, by a straightforward application of \cite[Sub-lemma 3.9]{tv}, is $m$-geometric for some $m$.
	\end{proof}
	As a special case of the stack defined in the lemma, we have, for $\gGG: * = \spec k \to \rpsperf(X)$, the derived stack of morphisms from $\fFF$ to $\gGG$ as the pullback of
	\[\begin{tikzcd}
		{\rmor(\fFF, \gGG)} \arrow{r} \arrow{d} & {*} \arrow{d}{\gGG}\\
		{\rpsmor(\fFF, -)} \arrow{r}{t} & {\rpsperf(X)}
	\end{tikzcd}\]
	or, equivalently, as the functor
	\begin{align*}
		\rmor(\fFF, \gGG): \cdga &\to \infgpd\\
		A &\mapsto Map_{\bf X_A}(\fFF_A, \gGG_A)
	\end{align*}
	Since all the derived stacks in the above cartesian square are locally geometric, we get, for a point $x: * \to \rmor(\fFF, \gGG)$ corresponding to a morphism $\phi: \fFF \to \gGG$, the following homotopy fiber sequence of tangent complexes
	\begin{align}\label{eq:targettriangle}
		\TT_{\rmor(\fFF, \gGG), x} \to \TT_{\rpsmor(\fFF, -), x} \xrightarrow{t} \TT_{\rpsperf(X), x}
	\end{align}
	All the statements above hold if we replace $\rperf(X)$ etc. by $\rperf(X)^{\leq 0}$ etc.
	
	\section{Derived quot functor}
	In this section we define the derived quot functor and prove Theorem \ref{maintheorem}. We need two definitions:
	\begin{defn}[{\cite[Lemma 2.2.2.2]{hag2}, \cite[Definition 3.15]{lowrey}}]
		For a cdga $A$, an $A$-dg-module $M$ is {\bf flat} if
		\begin{enumerate}
			\item $H^0(M)$ is flat over $H^0(A)$
			\item $M$ is {\bf strong} over $A$, i.e. the natural map
			\[ H^i(A)\otimes_{H^0(A)} H^0(M) \to H^i(M) \]
			is an isomorphism for all $i$.
		\end{enumerate}
		We say a quasicoherent sheaf $\fFF$ of $\oOO_{\dspec A}$-modules is {\bf flat over} $\dspec A$ if the associated $A$-module is. In particular, if $A$ is discrete (i.e. $A \simeq H^0(A)$) then a flat $A$-dg-module $M$ is itself discrete (i.e. $M \simeq H^0(M)$) with $H^0(M)$ a flat $A$-module.
		
		For a map $\pi:{\bf X} \to {\bf Y}$ of derived schemes, we say a quasicoherent $\oOO_{\bf X}$-module $\gGG$ is \textbf{flat over} $\bf Y$ if, for any discrete $\oOO_{\bf Y}$-module $\lL$, $\gGG \otimes^\LL \pi^*\lL$ is also discrete.
	\end{defn}
	\begin{defn}
		A map $\fFF \to \gGG$ in $QCoh({\bf Y})$ is called \textbf{surjective} if $\hHH^0(\fFF) \to \hHH^0(\gGG)$ is surjective.
	\end{defn}
	For the rest of this section, $\fFF$ is a fixed (discrete) coherent sheaf on $X$.
	\begin{defn}
		The \textbf{derived quot functor} $\dquot_\fFF$ is defined as the simplicial presheaf that maps $A \in \cdga$ to the full $\infty$-subgroupoid of $Perf({\bf X_A})_{\fFF_A/}^{\leq 0, \simeq}$ generated by the objects $\phi:\fFF_A \to \gGG$ such that
		\begin{enumerate}
			\item $\gGG$ is flat over $\dspec A$ and has proper support over $\dspec A$
			\item $\phi$ is surjective
		\end{enumerate}
	\end{defn}
	By definition, $\dquot_\fFF$ is a simplicial subpresheaf of $\rpsmor(\fFF, -)^{\leq 0}$.
	\begin{prop}
		The functor $\dquot_\fFF$ is a derived stack.
	\end{prop}
	\begin{proof}
		By construction, the simplicial sets of $\dquot_\fFF$ are full simplicial subsets of those of $\rpsmor(\fFF, -)^{\leq 0}$ i.e. for any cdga $A$, the inclusion \[\dquot_\fFF(A) \subset \rpsmor(\fFF, -)^{\leq 0}(A)\] identifies the former with a union of connected components of the latter. Therefore it will follow that $\dquot_\fFF$ is a derived stack once we prove the conditions of flatness and surjectivity are local for the étale topology. But that is clear since an étale hypercover is in particular a faithfully flat map of classical schemes that is also strong.
	\end{proof}
	
	\subsection{Geometricity}
	In this subsection, we will prove the geometricity of the derived stack $\dquot_\fFF$. First, we recall a result about open immersions of derived stacks that we will use often:
	
	\begin{prop}[{\cite[Proposition 2.1]{stv}}]\label{prop:stv}
		Let $\bf F$ be a derived stack and $t_0(\bf F)$ its truncation. There is a bijective correspondence
		\[ \phi_{\bf F}: \{ \text{ Zariski open substacks of } t_0({\bf F})\; \} \to \{ \text{ Zariski open derived substacks of } {\bf F}\; \} \]
		For any open substack $U_0 \hookrightarrow t_0({\bf F})$, we have a homotopy cartesian diagram of derived stacks
		\[\begin{tikzcd}
			{U_0} & {t_0({\bf F})} \\
			{\phi_{\bf F}(U_0)} & {\bf F}
			\arrow[hook, from=1-1, to=1-2]
			\arrow[hook, from=2-1, to=2-2]
			\arrow[from=1-1, to=2-1]
			\arrow[from=1-2, to=2-2]
		\end{tikzcd}\]
		where the vertical maps are the canonical closed immersions.
	\end{prop}
	\begin{lem}\label{lem:opensubstack}
		The map $\dquot_\fFF \to \rpsmor(\fFF, -)^{\leq 0}$ is a Zariski open immersion of derived stacks.
	\end{lem}
	\begin{proof}
		We establish the shorthand $\qQ := \dquot_\fFF$ and $\mM := \rpsmor(\fFF, -)^{\leq 0}$. By Proposition \ref{prop:stv}, it suffices to show that $t_0(\qQ)$ is an (ordinary) open substack of $t_0(\mM)$ and that (using the notation from the Proposition) $\phi_\mM(t_0(\qQ)) \simeq \qQ$. We conclude the former by noting that for a map $\spec A \to \mM$ corresponding to $\phi: \fFF_A \to \gGG$ where $\spec A$ is a classical affine scheme, the pullback map $t_0(\qQ)\times_{t_0(\mM)} \spec A \to \spec A$ is a Zariski open immersion of classical schemes, corresponding to the locus over which $\gGG$ is flat and $\phi$ is surjective. For the latter, we first recall the definition of $\phi_\mM(U)$ for an ordinary open substack $U \to t_0(\mM)$:
		\begin{align*}
			\cdga &\to \infgpd\\
			A &\mapsto \mM(A)\times_{t_0(\mM)(H^0(A))} U(H^0(A))
		\end{align*}
		and deduce that we need to show the equivalence
		\[ \qQ(A) \simeq \mM(A)\times_{\mM(H^0(A))} \qQ(H^0(A)) \]
		for all cdgas $A$. Since we have already shown that $\qQ(A) \subset \mM(A)$ is a union of connected components, it suffices to show that the above is a bijection at the level of 0-simplices. The 0-simplices of the right hand side are the maps $\phi:\fFF_A \to \gGG$ such that $\gGG_0 := \gGG \otimes_{\oOO_{\bf X_A}} \oOO_{\bf X_{H^0(A)}}$ is flat over $\spec H^0(A)$ and $\phi_0:\fFF_{H^0(A)} \to \gGG_0$ is surjective. There is a natural inclusion of $\qQ(A)_0$ into this set of 0-simplices since any map $\phi$ which has the flatness and surjectivity properties will induce these properties on $\gGG_0$ and $\phi_0$ respectively. It remains to check that if $\phi_0$ and $\gGG_0$ have these properties then so do $\phi$ and $\gGG$.
		
		Suppose $\phi_0$ and $\gGG_0$ have these properties. Then, since $\gGG_0$ is flat over $\spec H^0(A)$, it is discrete, and for any discrete sheaf $\lL$ over $\spec A$, $\gGG \otimes_{\oOO_{\bf X_A}} \pi^*\lL \simeq \gGG_0 \otimes_{\oOO_{\bf X_{H^0(A)}}} \pi^*\lL$ is also discrete, which implies that $\gGG$ is flat over $\dspec A$. Finally, the natural map $\gGG_0 \to \hHH^0(\gGG)$ is surjective and, composed with the surjective map $\phi_0: \hHH^0(\fFF_A) \simeq \fFF_{H^0(A)} \to \gGG_0$, gives the desired surjectivity of $\phi$.
	\end{proof}
	\begin{coro}\label{coro:quotisgeometric}
		The functor $\dquot_\fFF$ is a locally geometric derived stack.
	\end{coro}	
	
	\subsection{Truncation}
	In this subsection we prove parts (1) and (2) of Theorem \ref{maintheorem}. We begin with a definition. Ordinary commutative $k$-algebras embed into cdgas via an inclusion $i: kCAlg \to \cdga$. Given a functor
	\[ \textbf{F}: \cdga \to \infgpd \]
	its \textbf{truncation} is by definition the restricted functor
	\begin{align*}
		t_0(\textbf{F}): kCAlg &\to \infgpd\\
		A &\mapsto \textbf{F}(i(A))
	\end{align*}
	Denote by $Quot_\fFF$ the classical quot functor that assigns to an ordinary commutative $k$-algebra $A$ the ordinary groupoid
	\begin{align*} Quot_\fFF(A) = \{\, \text{quotients } &\fFF_A \twoheadrightarrow \gGG, \text{ where } \gGG \text{ is a quasicoherent } \oOO_{X\times \spec A} \text{-module} \\ &\text{ of finite presentation flat with proper support over } \spec A\, \}/ \sim\end{align*}
	where the equivalence $\sim$ identifies $\phi: \fFF_A\twoheadrightarrow \gGG$ and $\phi':\fFF_A \twoheadrightarrow \gGG'$ if there exists an isomorphism $\eta:\gGG \to \gGG'$ such that $\eta\circ\phi = \phi'$. We may think of this as a truncated $\infty$-groupoid. Such a quotient $\fFF_A \to \gGG$ is automatically an object of $\dquot_\fFF(A)$ and we get an induced map
	\[ \alpha: Quot_\fFF(A) \to t_0(\dquot_\fFF)(A) \]
	\begin{prop}
		The map $\alpha$ defined above is a bijection on $\pi_0$.
	\end{prop}
	\begin{proof}
		Injectivity follows from the fact that a morphism $(\fFF_A \to \pPP) \to (\fFF_A \to \pPP')$ of $\dquot_\fFF(A)$ is a quasi-isomorphism $\pPP \to \pPP'$ that makes the induced triangle commute up to homotopy, inducing an isomorphism and genuinely commutative triangle on $\hHH^0$.
		
		For surjectivity, suppose $\fFF_A \to \pPP$ is in some connected component of $t_0(\dquot_\fFF)(A)$. Then by flatness of $\pPP$ over $A$, it follows (for example, by passing to an affine open subscheme of $X \times \spec A$) that $\pPP$ and $\hHH^0(\pPP)$ are quasi-isomorphic and that $\hHH^0(\pPP)$ is flat over $\spec A$. The map $\hHH^0(\fFF_A) = \fFF_A \to \hHH^0(\pPP)$ defines the desired element of $Quot_\fFF(A)$.
	\end{proof}
	Let $x \in \dquot_\fFF(A)$ be an object corresponding to $q:\fFF_A \to \gGG$.
	\begin{prop}
		For $A \in kCAlg$, $x$ as defined above and $i > 0$, $\pi_i(\dquot_\fFF(A), x) = 0$.
	\end{prop}
	\begin{proof}
		By the previous proposition we may assume $\gGG$ is a genuine sheaf and that $q$ is surjective. The inclusion $\dquot_\fFF(A) \subset \rpsmor(\fFF, -)^{\leq 0, \simeq}(A)$ is an inclusion of connected components, so we can compute the homotopy groups in the latter. The map $Perf({ X_A})^{\leq 0}_{\fFF_A/} \to Perf({ X_A})^{\leq 0}$ is a left fibration (\cite[tag:018F]{kerodon}), therefore an isofibration (\cite[tag:01EW]{kerodon}), therefore induces (\cite[tag:01EZ]{kerodon}) a Kan fibration of cores 
		\[ S := Perf({ X_A})^{\leq 0, \simeq}_{\fFF_A/} \to Perf({ X_A})^{\leq 0, \simeq} =: P \]
		The fiber of this map over $\gGG \in Perf({ X_A})^{\leq 0, \simeq}$ is given by the mapping space
		\[ M := Map_{Perf({ X_A})}(\fFF_A, \gGG) \]
		and we have a long exact sequence
		\begin{align*}
			&\cdots \to \pi_{n+1}(P, \gGG) \to \pi_n(M, x) \to \pi_n(S, x) \to \pi_n(P, \gGG) \to\\
			&\cdots \to \pi_1(P, \gGG) \to \pi_0(M) \to \pi_0(S) \to \pi_0(P)
		\end{align*}
		Since $\fFF_A$ and $\gGG$ are both discrete (recall $A$ is discrete), it follows that the mapping space $M$ can be computed in the heart of the stable $\infty$-category $Perf({X_A})$ (we use here the standard t-structure). This heart is equivalent to the nerve of the ordinary category, $Coh({ X_A})$ (\cite[Proposition 7.1.1.13(3)]{HA}) and therefore we have
		\begin{align*}
			\pi_0(M) &= \Hom_{D({X_A})}(\fFF_A, \gGG)\\
			\pi_i(M, x) &= 0, \qquad i>0
		\end{align*}
		The higher homotopy groups of $Perf({ X_A})$ (which are the same as those of $PsPerf({ X_A})$ since the latter is a union of connected components of the former) have been computed in \cite[\S 1.3.7]{hag2}:
		\begin{align*}
			\pi_1(P, \gGG) &= Aut_{D({X_A})}(\gGG)\\
			\pi_i(P, \gGG) &= Ext^{1-i}_{D({X_A})}(\gGG, \gGG), \qquad i > 1
		\end{align*}
		For a discrete sheaf $\gGG$, the negative self ext groups are zero. Putting it all together, we conclude that
		\begin{align*}
			\pi_i(\dquot_\fFF(A), x) = \pi_i(Map_{Perf({X_A})}(\fFF_A, \gGG), x) = 0, \qquad i > 1
		\end{align*}
		and that
		\[ 0 \to \pi_1(\dquot_\fFF(A), x) \to \pi_1(Perf({X_A})^{\leq 0, \simeq}, \gGG) \xrightarrow{b} \pi_0(Map_{Perf({X_A})}(\fFF_A, \gGG))\]
		is an exact sequence. Therefore we will be done if we show that the map $b$ above is injective. Using the above characterizations of homotopy groups, this map is the same as the map
		\begin{align*}
			b: Aut_{D({X_A})}(\gGG) &\to \Hom_{D({X_A})}(\fFF_A, \gGG)\\
			f &\mapsto f\circ q
		\end{align*}
		which is injective by surjectivity of $q$.
	\end{proof}
	Combining the above two propositions, we conclude the map
	\[ \alpha: Quot_\fFF \to t_0(\dquot_\fFF) \]
	is an equivalence of ordinary (higher) stacks. Our quasi-projectivity assumption on $X$ implies (due to Grothendieck, for example \cite[Theorem 6.3]{nitsure}) that $Quot_\fFF$ is representable by a scheme.
	\begin{lem}[{\cite[Theorem C.0.9]{hag2}}]
		If $\mathbf{F}$ is an $n$-geometric $D^-$-stack such that $t_0(\mathbf{F})$ is an Artin $(k+1)$-stack, then $\mathbf{F}$ is $k$-geometric.
	\end{lem}
	In particular, if $\mathbf{F}$ is an $n$-geometric stack such that $t_0(\mathbf{F})$ is a scheme, then $\mathbf{F}$ is a derived scheme. Therefore we conclude
	\begin{prop}
		The derived stack $\dquot_\fFF$ is a derived scheme.
	\end{prop}
	This concludes the proof of parts (1) and (2) of Theorem \ref{maintheorem}.

	\subsection{Tangent complex}
	The derived scheme $\dquot_\fFF$ has a connective cotangent complex. In this section we compute its tangent complex at a point $x:{\bf \ast} = \dspec k \to \dquot_\fFF$. From the isomorphism $\pi_0(\dquot_\fFF(k)) \simeq Quot_\fFF(k)$, we may assume that $x$ corresponds to an actual quotient $\phi: \fFF_k = \fFF \twoheadrightarrow \gGG$ of sheaves on $X$. Let $\mathcal{K}$ be the kernel of this quotient i.e. we have an exact sequence
	\[ 0 \to \mathcal{K} \to \fFF \xrightarrow{\phi} \gGG \to 0 \]
	which gives rise to the exact triangle
	\begin{align}\label{eq:les}
		\RHom_X(\fFF, \gGG) \to \RHom_X(\mathcal{K}, \gGG) \to \RHom_X(\gGG, \gGG)[1] \to \RHom_X(\fFF, \gGG)[1]
	\end{align}
	in $D(X)$.
	Since $\dquot_\fFF \to \rpsmor(\fFF, -)$ is an open immersion, we may compute the tangent complex in the latter stack instead. Recall the homotopy fiber sequence (\ref{eq:targettriangle}) of tangent complexes:
	\begin{align*}
		\TT_{\rmor(\fFF, \gGG), x} \to \TT_{\rpsmor(\fFF, -), x} \xrightarrow{t} \TT_{\rpsperf(X), x}
	\end{align*}
	This gives an exact triangle in $D(k)$:
	\[
	\TT_{\rmor(\fFF, \gGG), x} \to \TT_{\rpsmor(\fFF, -), x} \xrightarrow{t} \TT_{\rpsperf(X), x} \to \TT_{\rmor(\fFF, \gGG), x}[1]
	\]
	Using the characterizations of tangent complexes from \cite{tv}, this is the triangle
	\[ \RHom_X(\fFF, \gGG) \to \TT_{\dquot_\fFF, x} \to \RHom_X(\gGG, \gGG)[1] \to \RHom_X(\fFF, \gGG)[1] \]
	The rightmost map is given by composition with $\phi$, so comparison with (\ref{eq:les}) and the uniqueness of the dg-cone construction (\cite[Proposition 4.3]{thesisdgcats}; also \cite[Remark 1.1.1.7]{HA}) gives us the isomorphism
	\[ \TT_{\dquot_\fFF, x} \simeq \RHom_X(\mathcal{K}, \gGG) \]
	in $D(k)$. This proves Part (3) of Theorem \ref{maintheorem} and concludes the proof of the theorem.

	\subsection{Stratification}
	For $X$ a smooth quasiprojective variety and $\eEE$ a coherent sheaf on $X$, there is a stratification (which depends on the choice of very ample line bundle on $X$ which we suppress for now)
	\[ Quot_\eEE = \coprod_{\Psi \in \QQ[t]}  Quot^\Psi_\eEE \]
	where the $\Psi$'s are the Hilbert polynomials of the quotients. Using Proposition \ref{prop:stv}, we get unique derived extensions of each of the components in the decomposition above and we have a stratification
	\[ \dquot_\eEE = \coprod_{\Psi \in \QQ[t]}  \dquot_\eEE^\Psi \]
	From the description of the map $\phi_{\bf F}$ in the proposition, we see that $\dquot_\eEE^\Psi$ is the stack whose $A$-points are the maps $\eEE_{\bf A} \to \gGG$ in $\dquot_\eEE(A)$ such that $\hHH^0(\gGG) \simeq \gGG_0 = \gGG \otimes^\LL_{\oOO_{\dspec A}} \oOO_{\spec H^0(A)}$ has Hilbert polynomial $\Psi$ when restricted to a fiber of the projection $X \times \spec H^0(A) \to \spec H^0(A)$. More precisely (as in the classical case), this is the Hilbert polynomial of an extension of $\gGG_0$ to a projective completion of the fiber. This definition will be useful in the next section when we address derived quot schemes of points.
	
	\section{Quot scheme of points}
	Denote by $\dquot^n_X$ the derived Quot scheme $\dquot_{\oOO_X}^P$ where $P$ is the constant polynomial $n$.
	\begin{eg}\label{eg:quotn}
		Let $R = k[x_1, \ldots, x_m]/(f_1, \ldots, f_r)$ be a complete intersection $k$-algebra such that $X = \spec R$ is a nonsingular affine variety. The derived quot functor $\dquot^n_X$ assigns (by definition) to a cdga $A$ the $\infty$-groupoid whose objects are maps $R \otimes^\LL A \to Q$ such that the induced map $R \otimes H^0(A) \simeq H^0(R \otimes^\LL A) \to H^0(Q)$ is surjective, where $Q$ is an $R \otimes^\LL A$-dg-module, flat over $A$ and such that the $H^0(A)$-module $H^0(Q)$ is projective of rank $n$. These last two conditions, along with \cite[Proposition 2.22(6)]{tv}, imply that $Q$ itself is a projective $A$-module of rank $n$.
		Working étale locally on $A$ if needed, we can mimic the proof of that proposition to lift the isomorphism $H^0(A)^n \to H^0(Q)$ (since $R \otimes H^0(A) \to H^0(Q)$ is a point in $Quot^n_X(H^0(A)) \simeq Hilb^n_X(H^0(A))$) to a quasi-isomorphism $A^n \simeq Q$.
		Thus we conclude that $\dquot^n_X(A)$ is étale locally a quotient of the simplicial subset of
		\[ Map_{dga}(R, End_A(A^n)) \times A^n \]
		generated by those 0-simplices such that the induced map is surjective. The mapping space in the above product can be computed by taking a cofibrant replacement of $R$ in the Dwyer-Kan model structure on the category of dg-categories (first introduced in \cite{hinich}), so we have
		\begin{align*}
			Map_{dga}(R, End_A(A^n)) &\simeq Hom_{dga}(\cofib R, \gl n \otimes A)\\
			&\simeq Hom_{cdga}(\cofib R_n, A)
		\end{align*}
		where $\cofib{R}$ is the semi-free dga resolution of $R$ with generators
		\begin{align*}
			x_i \qquad &\text{in degree 0, with } dx_i = 0\\
			a_{ij}, s_i \qquad &\text{in degree -1, with } da_{ij} = [x_i, x_j], ds_i = f_i(x_1, \ldots, x_m)\\
			b_{n_j} \qquad &\text{in degree }j, j < -1
		\end{align*}
		where the differentials of the $b_{n_j}$'s are expressions involving the generators of higher degree and the partial derivatives $\partial f_i/\partial x_j$ (we provide an example of such a differential in Example \ref{eg:fermat}). Here a choice of ordering on the monomials of the $f_i$'s has been made, but is ultimately irrelevant (up to quasi-isomorphism) since, given a different ordering, the corresponding permutation can be used to construct a quasi-isomorphism.
		Then, as in \cite{berest}, $\cofib{R}_n$ is the semi-free cdga generated by the entries of the $n \times n$ matrices
		\begin{align*}
			X_i \qquad &\text{in degree 0, with } dX_i^{\mu\nu} = 0\\
			A_{ij}, S_i \qquad &\text{in degree -1, with } dA_{ij}^{\mu\nu} = [X_i, X_j]^{\mu\nu}, dS_i^{\mu\nu} = f_i(X_1, \ldots, X_m)^{\mu\nu}\\
			B_{n_j} \qquad &\text{ in degree } j, j <-2
		\end{align*}
		
		Since $A^n \simeq Hom_{cdga}(k[y_1, \ldots, y_n], A)$, we have that
		\[  Map_{dga}(R, End_A(A^n)) \times A^n \simeq Hom_{cdga}(\cofib{R}_n \otimes^\LL k[y_1, \ldots, y_n], A) \]
		
		Next we briefly recall the construction of $Hilb^n(X)$ as a quotient of the scheme of commuting matrices (using, for example, \cite{hennijardin} as reference). Define the variety
		\[ V(R, n) := \left\{\ \left(X_1, \ldots, X_m, v\right) \in \gl{n}^m \times \CC^n\ \middle|\ [X_i, X_j] = 0, f_i(X_1, \ldots, X_m) = 0\  \right\} \]
		and let $V(R, n)^{st} \subset V(R, n)$ be the open subset consisting of those points $(X_1, \ldots, X_m, v)$ such that the induced map
		\begin{align*}
			R &\to \CC^n\\
			x_i &\mapsto X_i\; v
		\end{align*}
		is surjective. The $G := GL_n$ action on $V(R, n)^{st}$ given by
		\[ g\cdot(X_1, \ldots, X_m, v) = (gX_1g^{-1},\ldots, gX_mg^{-1}, gv) \]
		makes $\mMM(R, n) := V(R, n)^{st}/G$ a good quotient and we have an isomorphism
		\[ Hilb^n(X) \simeq \mMM(R, n). \]
		Now we set
		\[ {\bf V}(R, n) := \dspec (\cofib{R}_n \otimes^\LL k[y_1, \ldots, y_n]) = \dspec(\cofib{R}_n) \times \CC^n \]
		Then it follows that $t_0({\bf V}(R, n)) \simeq V(R, n)$.
		
		Next we define ${\bf V}(R, n)^{st}$ to be the derived subscheme of ${\bf V}(R, n)$ consisting of the 0-simplices in ${\bf V}(R, n)(A)$ for which the induced map $R \otimes H^0(A) \to H^0(A^n)$ is surjective. But these are precisely the 0-simplices for which the induced map in $V(R, n)(H^0(A))$ lies in $V(R, n)^{st}(H^0(A))$. In other words, using the notation of Proposition \ref{prop:stv},
		\[ \phi_{{\bf V}(R, n)} (V(R, n)^{st})  \simeq {\bf V}(R, n)^{st} \]
		The action of $G$ lifts to an action on ${\bf V}(R, n)^{st}$ and we can construct the quotient derived stack $[{\bf V}(R, n)^{st}/G]$. The fiber of the map
		\begin{align*}
			{\bf V}(R, n)^{st}(A) &\to \dquot^n(X)(A)\\
			(R \otimes^\LL A \to A^n) &\mapsto (\oOO_{\dspec (R\otimes^\LL A)} \to \oOO_{\dspec A}^n)
		\end{align*}
		consists of the automorphisms of $\oOO_{\dspec (R\otimes^\LL A)} \to \oOO_{\dspec A}^n$ in the undercategory $Perf({\bf X_A})_{\oOO_{{\bf X}_A/}}^{\leq 0, \simeq}$.
		\begin{prop}
			There is an equivalence of derived stacks
			\[ [{\bf V}(R, n)^{st}/G] \simeq \dquot^n(X) \]
		\end{prop}
		\begin{proof}
			The induced map on homotopy sheaves
			\[ \pi_0\left({\bf V}(R, n)^{st}\right) \to \pi_0\left(\dquot^n(X)\right) \]
			is an epimorphism (of sheaves of sets) since étale locally any projective module is free. Therefore, by \cite[Lemma 1.3.4.3]{hag2}, $\dquot^n(X)$ is equivalent to the homotopy nerve of the map ${\bf V}(R, n)^{st} \to \dquot^n(X)$. On the other hand, by \cite[\S 3]{dag-ems}, for a group action of $G$ on a derived scheme $\bf Y$ the quotient stack $[{\bf Y}/G]$ is the homotopy colimit of the simplicial object $B({\bf Y}, G)$ that in degree $k$ is given by ${\bf Y}\times G^k$. Since both simplicial objects under consideration are Segal groupoids, it suffices to construct an equivalence
			\[ {\bf V}(R, n)^{st} \times G \simeq {\bf V}(R, n)^{st}\times^h_{\dquot^n(X)} {\bf V}(R, n)^{st}. \]
			Given a cdga $A$, $({\bf V}(R, n)^{st} \times G)(A)$ is a product of the stable locus of $Map_{dga}(R, End_A(A^n)) \times A^n$ with $GL_n(A)$, while $({\bf V}(R, n)^{st}\times^h_{\dquot^n(X)} {\bf V}(R, n)^{st})(A)$ is the simplicial subset of ${\bf V}(R, n)^{st}(A)\times^h {\bf V}(R, n)^{st}(A)$ generated by 0-simplices $(R \otimes^\LL A \xrightarrow{\phi} A^n, R \otimes^\LL A \xrightarrow{\psi} A^n)$ for which $\phi$ and $\psi$ map to equivalent objects in $\dquot^n(A)$. Thus the map
			\begin{align*}
				({\bf V}(R, n)^{st} \times G)(A) &\to ({\bf V}(R, n)^{st}\times^h_{\dquot^n(X)} {\bf V}(R, n)^{st})(A)\\
				(R \otimes^\LL A \xrightarrow{\phi} A^n, M) &\mapsto (R \otimes^\LL A \xrightarrow{\phi} A^n, R \otimes^\LL A \xrightarrow{M\circ\phi} A^n)
			\end{align*}
			achieves the desired equivalence.
		\end{proof}
	\end{eg}
	Thus the derived quot scheme of points has a presentation as a group quotient which is just a derived enhancement of the well-known one for the classical quot scheme of points.
	
	\section{Shifted symplectic structure}
	In this section we describe a shifted form on the derived quot scheme and conclude with an explicit calculation.
	
	\subsection{Background}
	First we recall a few notions from \cite{ptvv}, using the descriptions in \cite{bbbbj}. Let $\bf X$ be a locally geometric derived stack. Then $\bf X$ has a global cotangent complex $\LL_{\bf X} \in Perf(\bf X)$, with dual the tangent complex $\TT_{\bf X}$. One can define the exterior powers $\Lambda^p\LL_{\bf X}$ for $p = 0, 1, \ldots$. Regard $\Lambda^p\LL_{\bf X}$ as a complex, with differential $d$:
	\[ \cdots \xrightarrow{d} \left(\Lambda^p\LL_{\bf X}\right)^{k-1} \xrightarrow{d} \left(\Lambda^p\LL_{\bf X}\right)^{k} \xrightarrow{d} \left(\Lambda^p\LL_{\bf X}\right)^{k+1} \xrightarrow{d} \cdots \]
	There are de Rham differentials $d_{dR}:\Lambda^p\LL_{\bf X} \to \Lambda^{p+1}\LL_{\bf X}$ with $d_{dR}\circ d_{dR} = d\circ d_{dR} + d_{dR}\circ d = 0$.
	\begin{defn}
		A \textbf{$k$-shifted $p$-form} on $\bf X$ is an element $\omega^0 \in \left(\Lambda^p\LL_{\bf X}\right)^{k}$ with $d\omega^0 = 0$. For $p = 2$, this induces by adjunction a map $\omega^0:\TT_{\bf X} \to \LL_{\bf X}[k]$.
		
		A \textbf{$k$-shifted closed $p$-form} on $\bf X$ is a sequence
		\[ \omega = \left(\omega^0, \omega^1, \ldots \right), \qquad \omega^i \in \left(\Lambda^{p+i}\LL_{\bf X}\right)^{k-i} \]
		with $d\omega^0 = 0$ and $d_{dR}\omega^i + d\omega^{i+1} = 0$ for $i \geq 0$.
		
		A \textbf{$k$-shifted symplectic structure} on $\bf X$ is a $k$-shifted closed 2-form $\left(\omega^0, \omega^1, \ldots \right)$ on $\bf X$ whose induced morphism $\omega^0:\TT_{\bf X} \to \LL_{\bf X}[k]$ is an equivalence.
	\end{defn}
	The main theorem about shifted symplectic structures that we will need is a consequence of \cite[Theorem 2.5]{ptvv}. Let $X$ be a smooth and proper Calabi-Yau $d$-fold equipped with a trivialization $\omega_X \simeq \oOO_X$.
	\begin{thm}[\cite{ptvv}]
		The derived stack of perfect complexes $\rperf(X)$ on $X$ is canonically endowed with a $(2-d)$-shifted symplectic structure.
	\end{thm}
	Since we work in the slightly more general setting of quasiprojective varieties, we will also need the following theorem:
	\begin{thm}[\cite{bravdyck}]
		Let $X$ be a finite type Gorenstein scheme of dimension $d$ with a trivialization $\omega_X \simeq \oOO_X$ of its canonical bundle. Then the derived stack $\rpsperf(X)$ of perfect complexes with proper support on $X$ has an induced $(2-d)$-shifted symplectic structure.
	\end{thm}
	We note that the former is a special case of the latter, and in both cases the shifted symplectic structure
	\[ \TT_{\rpsperf(X)} \wedge \TT_{\rpsperf(X)} \to k[2-d] \]
	is given by the trace map
	\begin{align*}
		\RR\Hom_X(\gGG, \gGG)[1] \wedge \RR\Hom_X(\gGG, \gGG)[1] &\to k[2-d]\\
		A \wedge B &\mapsto Tr(A\circ B)
	\end{align*}
	
	\subsection{Pullback of the shifted symplectic form to the derived quot scheme}
	Let $X$ be a smooth quasiprojective Calabi-Yau $d$-fold and $\fFF$ a coherent sheaf on $X$. By the results of this section, the derived stack $\rpsperf(X)$ has a canonical $(2-d)$-shifted symplectic structure. Using the map
	\[ \dquot_\fFF \to \rpsperf(X) \]
	we can pull back this shifted symplectic structure to get a closed $(2-d)$-shifted 2-form on $\dquot_\fFF$.

	\begin{eg}\label{eg:fermat}
		Let $f(w,x,y,z) = w^5 + x^5 + y^5 + z^5 + 1$ and 
		\[ R := k[w,x,y,z]/(f) \]
		Recall that the Fermat quintic $Y$ is defined to be the vanishing locus of
		\[ g(x_0, \ldots, x_4) = x_0^5 + \ldots + x_4^5 \]
		in $\PP^4$. Consequently, $X := \spec R$ is the open subscheme of $Y$ defined by $x_0 \neq 0$, with the change of variables
		\[w = x_1/x_0 \qquad x = x_2/x_0 \qquad y = x_3/x_0 \qquad z = x_4/x_0. \]
		Using $d$ to mean $d_{dR}$ in the rest of this paragraph for convenience, the Calabi-Yau form $\Omega$ on Y is defined as
		\[ \Omega = Res\left(\frac{\Theta}{g}\right), \qquad \text{ where } \Theta = \sum_{i=0}^4 (-1)^ix_idx_0\wedge\ldots\wedge \widehat{dx_i} \wedge \ldots \wedge dx_4\]
		One may verify that
		\[ \Theta\vert_Y = dg \wedge \alpha, \qquad \text{ where } \alpha = \frac{\sum_{i=1}^4 (-1)^ix_idx_1\wedge\ldots\wedge \widehat{dx_i} \wedge \ldots \wedge dx_4}{\partial g / \partial x_0},\]
		so that on the open set where $\partial g / \partial x_0 = 5x_0^4 \neq 0$, $\Omega = \alpha$ (up to a scalar multiple) since $\alpha$ is holomorphic. Applying our change of variables, we may assume that the Calabi-Yau form on $X$ is given by
		\[ wdx\wedge dy\wedge dz -xdw\wedge dy\wedge dz + ydw\wedge dx\wedge dz - zdw\wedge dx\wedge dy\]
		which is a constant multiple (reusing $\Omega$ by abuse of notation) of
		\begin{align*}
			\Omega &= \left(wdx - xdw\right)\wedge dy\wedge dz + \left(wdy - ydw\right)\wedge dz\wedge dx + \left(wdz - zdw\right)\wedge dx\wedge dy\\
			&+ \left(ydz - zdy\right)\wedge dw\wedge dx + \left(zdx - xdz\right)\wedge dw\wedge dy + \left(xdy - ydx\right)\wedge dw\wedge dz.
		\end{align*}
		
		We resume using $d$ for cdga differentials and $d_{dR}$ for the de Rham differential. Let $D = k\langle w, x, y, z\rangle$ be the free associative unital algebra on four generators. The semi-free dga resolution $\cofib R$ is, as in Example \ref{eg:quotn}, generated by
		\begin{align*}
			\ldots \to D\langle v_w, \ldots, t_w, \ldots \rangle \to &D\langle u_{wx}, \ldots, s\rangle \to &D\\
			&u_{wx} \mapsto &[w,x]\\
			&s \mapsto &f\\
			v_w \mapsto &[x, u_{yz}] + [y, u_{zx}] + [z, u_{xy}]\\
			t_w \mapsto &[w,s] + \sum_{i=0}^4 \left(x^i u_{wx}x^{4-i} + y^i u_{wy}y^{4-i} + z^i u_{wz}z^{4-i}\right)
		\end{align*}
		Then the cdga $\cofib R_n$ is generated by the entries of the $n \times n$ matrices
		\begin{align*}
			W, X, Y, Z \qquad &\text{in degree 0, with } dW^{\mu\nu} = \ldots = 0\\
			U_{wx}, \ldots, S \qquad &\text{in degree -1, with } dU_{wx}^{\mu\nu} = [W, X]^{\mu\nu}, \ldots, dS^{\mu\nu} = (W^5+X^5+Y^5+Z^5 + I)^{\mu\nu}\\
			V_{w}, \ldots, T_w, \ldots &\text{in degree -2, with } dV_w^{\mu\nu} = [X, U_{yz}] + [Y, U_{zx}] + [Z, U_{xy}],\ldots,\\
			& dT_w^{\mu\nu} = [W,S]^{\mu\nu} + \sum_{i=0}^4 \left(X^i U_{wx}X^{4-i} + Y^i U_{wy}Y^{4-i} + Z^i U_{wz}Z^{4-i}\right)^{\mu\nu}, \ldots\\
			B_{n_j} \qquad &\text{ in degree } j, j <-2
		\end{align*}
		with
		\[ \spec H^0(\cofib R_n) \simeq V(R, n) \]
		We have the derived open subscheme ${\bf V} (R, n)^{st} \to {\bf V}(R, n)$ with the map ${\bf V} (R, n)^{st} \to \dquot^n_X$ that realizes the latter as a quotient of the former by a group action. Under this map we can pull back the closed shifted form to ${\bf V} (R, n)^{st}$.
		We now describe a closed $(-1)$-shifted 2-form on ${\bf E} := \dspec(\cofib{R}_n)$.
		Let
		\begin{align*}
			\Phi &= \left(Wd_{dR}X - Xd_{dR}W\right)U_{yz} + \left(Wd_{dR}Y - Yd_{dR}W\right)U_{zx} + \left(Wd_{dR}Z - Zd_{dR}W\right)U_{xy}\\
			&+ \left(Yd_{dR}Z - Zd_{dR}Y\right)U_{wx} + \left(Zd_{dR}X - Xd_{dR}Z\right)U_{wy} + \left(Xd_{dR}Y - Yd_{dR}X\right)U_{wz}
		\end{align*}
		and
		\begin{align*}
			\phi &:= tr(\Phi)\\
			\omega^0 &:= d_{dR}\phi\\
			\omega &:= (\omega^0, 0, 0, \ldots)
		\end{align*}
		Then, using the cyclic property of the trace and the identity $d\circ d_{dR} + d_{dR}\circ d = 0$, it can be verified that $d\omega^0 = 0$, so that $\omega$ is a closed $(-1)$-shifted 2-form.
	\end{eg}
	
	\begin{conj}\label{conjecture}
		The $(-1)$-shifted 2-form $\omega$ described above is equivalent to the one pulled back from $\rpsperf(X)$ under the composition
		\[ \bf E \to \dquot^n_X \to \rpsperf(X) \]
	\end{conj}
	Further considerations about this shifted symplectic structure will appear in future work.

	\printbibliography
	
\end{document}